\documentclass[pdflatex,sn-apa]{sn-jnl}% APA Reference Style
\usepackage{graphicx}%
\usepackage{multirow}%
\usepackage{amsmath,amssymb,amsfonts}%
\usepackage{amsthm}%
\usepackage{mathrsfs}%
\usepackage[title]{appendix}%
\usepackage{xcolor}%
\usepackage{textcomp}%
\usepackage{manyfoot}%
\usepackage{booktabs}%
\usepackage{algorithm}%
\usepackage{algorithmicx}%
\usepackage{algpseudocode}%
\usepackage{listings}%
\theoremstyle{thmstyleone}%
\newtheorem{theorem}{Theorem}%  meant for continuous numbers
\theoremstyle{thmstyletwo}%

\theoremstyle{thmstylethree}%

\begin{document}

\title[Didactic analysis of the modality of study of the real numbers ]{Didactic analysis of the modality of study of the real numbers in the Degree in Mathematics}

%%=============================================================%%
%% GivenName	-> \fnm{Joergen W.}
%% Particle	-> \spfx{van der} -> surname prefix
%% FamilyName	-> \sur{Ploeg}
%% Suffix	-> \sfx{IV}
%% \author*[1,2]{\fnm{Joergen W.} \spfx{van der} \sur{Ploeg} 
%%  \sfx{IV}}\email{iauthor@gmail.com}
%%=============================================================%%

\author*[1]{\fnm{Jos\'e Gin\'es} \sur{Esp\'in}}\email{josegines.espin@upct.es}

\author[2]{\fnm{Josep} \sur{Gasc\'on}}\email{josep.gascon@uab.cat}

\author[3]{\fnm{Pedro} \sur{Nicol\'as}}\email{pedronz@um.es}
%\equalcont{These authors contributed equally to this work.}

\affil*[1]{\orgname{Universidad Polit\'ecnica de Cartagena}, \country{Spain}}

\affil[2]{\orgname{Universitat Aut\`onoma de Barcelona}, \country{Spain}}

\affil[3]{\orgname{Universidad de Murcia}, \country{Spain}}

%%==================================%%
%% Sample for unstructured abstract %%
%%==================================%%

\abstract{To conduct a didactic analysis of the current modality of study of real numbers in the Bachelor's Degree in Mathematics, we rely on the notion of didactic paradigm and employ a novel research methodology within the framework of the Anthropological Theory of the Didactic. This methodology proposes that, in general, analysing a modality of study of a field of knowledge, currently implemented in a certain didactic institution, consists of modelling, through respective didactic paradigms, not only the current modality of study but also: (1) an alternative modality of study (a possible future) whose didactic ends include attempting to avoid certain current phenomena considered ``undesirable'' from the perspective of research; and (2) an antecedent modality of study (a possible past) that allows us to explain the emergence of current phenomena as a reaction to the existence of certain antecedent phenomena.}


%%\pacs[JEL Classification]{D8, H51}

%%\pacs[MSC Classification]{35A01, 65L10, 65L12, 65L20, 65L70}

\maketitle

\section{Introduction: didactics as a science of study}\label{sec1}

From the perspective of the Anthropological Theory of the Didactic (ATD), we understand ``the didactic'' as relating to study in a broad sense, including in the notion of study not only the processes aimed at enabling its dissemination, but also those intended to facilitate its application or use in specific situations, the processes of creating and reorganizing knowledge, the processes of adaptation to different institutions, and those that facilitate access to new problems. In this sense, we consider didactics to be \textit{the science of study} \citep{eslabon, gascon1997} and of the modalities of study, including their impact on learning. If we restrict ourselves to the case of mathematical knowledge, we obtain a definition of the didactics of mathematics that generalizes the one proposed by Guy Brousseau in a lecture given in February 1993 at the Autonomous University of Barcelona:

\begin{quote}
	Didactics of mathematics is the science of the conditions for the creation and diffusion of mathematical pieces of knowledge. \citep[p. 33, our translation]{brousseau}.
\end{quote}

Given a \textit{field of knowledge}, $\mathcal{F}$, designated as the object of study\footnote{A didactic institution in which  a modality of study of the \textit{inquiry-based learning} type is implemented may also propose, as a subject of study, a \textit{question} formulated using the concepts and terms provided by a system, that is, by a part of the world (mathematical or extra-mathematical) about which it seeks to increase knowledge. The ultimate reference point for knowledge is, in all cases, a system or type of systems.} in a \textit{didactic institution}, $I$, we aim to investigate how a \textit{modality of study} of $\mathcal{F}$ in $I$ can be characterized. In short, we take the current (or possible) modalities of study in didactic institutions as a type of empirical system to be investigated and, to this end, we construct theoretical models of these modalities, which we call \textit{didactic paradigms}. We assume that didactics, like any empirical science, investigates certain systems indirectly, that is, through the intermediation of models constructed for that purpose.

In what follows, we will begin by analysing the basic notions related to didactic paradigms and describe in some detail the structure of \textit{praxeologies} \citep{praxeologies} because they play an important role in interpreting the epistemological model of a field of knowledge. Next, we will discuss the need to establish a \textit{reference didactic paradigm} from which to carry out the didactic analysis (\textit{economic} and \textit{ecological}\footnote{We understand the \textit{economy} of a system to be the coordination of the components (or subsystems) involved in its functioning \citep{moliner2007}. By the \textit{ecology} of a system, we mean the set of conditions that affect (or have affected) its change. These conditions explain, in part, its past evolution, its current behaviour (its economy), and its possible future developments. We will clarify these concepts in Section~\ref{sec3} and in in Sections~\ref{sec5} and~\ref{sec6} respectively.}) of the current modality of study.

To analyse the economy of the modality of study of real numbers, we will begin by summarizing a praxeological analysis of the \textit{current epistemological model} in the Degree in Mathematics\footnote{Throughout the paper, we will use \textit{Degree in Mathematics} to refer to the undergraduate \textit{Bachelor's Degree in Mathematics}.} (DM). This analysis, which will be carried out from the perspective of a certain \textit{reference epistemological model} (REM), will detect praxeological features that suggest the emergence of ``undesirable'' didactic phenomena (when interpreted from the perspective of ATD) and that can be formulated in terms of didactic means and didactic ends. At this point, the research problem can be explicitly stated.  

The explicit construction of the REM and the reference didactic paradigm based on it constitutes a prospective analysis of the current modality of study that initiates the ecological analysis and will allow us to formulate the hypothesis that a modality of study of real numbers governed by this paradigm would avoid the aforementioned undesirable didactic phenomena. To complete the ecological analysis, we construct a hypothetical \textit{antecedent didactic paradigm} that will supposedly help to `explain' the emergence and persistence of these phenomena. We will then be in a position to formulate new research problems based on other possible didactic analyses of the modality of study. We will conclude with a brief section of conclusions in which we emphasise the importance of the ecological dimension of the problem of paradigm shift and the fundamental role that didactic phenomena play in such changes.

\section{Theoretical framework: structure and dynamics of didactic paradigms}\label{sec2}

Didactic paradigms are theoretical models of the modalities of study of a field of knowledge in a didactic institution. When we speak of a \textit{field of knowledge}, we refer to what didactic institutions usually designate as the object of study, such as real numbers, Aristotelian ethics, or the simple pendulum. The notion of \textit{didactic institution} refers to a social institution with the stated aim of studying certain fields of knowledge. A \textit{social institution} consists of a set of constitutive rules, stated by convention, which establish: a series of institutional positions; permitted, mandatory and prohibited actions for each position; and rewards and punishments for certain actions \citep{searle2010}. Examples of social institutions are chess, language, marriage, law, employment and religion. Examples of didactic institutions are Primary Education, the Degree in Physics or Research in Mathematics. Social sciences seek to understand how social institutions work. In didactics, we want to understand how didactic institutions work.

\subsection{Structure of didactic paradigms}\label{subsect2.1}

Given a didactic institution, $I$, and a field of knowledge, $\mathcal{F}$, we seek to understand how the current modality of study of $\mathcal{F}$ in $I$ is organized and how it works. To this end, we shall employ the notion of didactic paradigm as a model. Specifically, we say that a \textit{didactic paradigm (for the study) of} $\mathcal{F}$ \textit{in} $I$ is a system consisting of three subsystems: $$\text{DP}_I(\mathcal{F})=[\text{EM}_I(\mathcal{F}), \text{DE}_I(\mathcal{F}), \text{DM}_I(\mathcal{F})],$$ where:

\begin{itemize}
	\item $\text{EM}_I(\mathcal{F})$ is the \textit{epistemological model} of $\mathcal{F}$ in $I$, that is, the answer provided by $I$ to the question: ``what does $\mathcal{F}$ consist of?'', and that we describe in terms of praxeologies (see Section~\ref{subsection2.3}),
	\item $\text{DE}_I(\mathcal{F})$ are the \textit{didactic ends}, that is, the answer to the question: ``what is the reason for studying $\mathcal{F}$ in $I$?'', 
	\item $\text{DM}_I(\mathcal{F})$ are the didactic means, that is, the actions (which may involve the use of certain objects) carried out to achieve the didactic ends. They constitute the answer to the question: ``how is $\mathcal{F}$ studied in $I$?''. 
\end{itemize}

We thus have a structural description of didactic paradigms, which can be intuitively interpreted as the set of ideas shared by a given study community that strongly shape the way its members act and interpret their work. Given the interdefinability and mutual dependence of the functions of its components within the overall system, the functional unity of a didactic paradigm (DP) becomes evident; consequently, it must be considered a complex system \citep{garcia2006}. We know that what is studied is inseparable from how it is studied, but we must not forget that both are inseparable from what it is studied for and, as we shall see below, from the phenomena to which the modality of study in question responds (the \textit{raison d'\^etre} for the emergence of this modality of study).

\subsection{Dynamics of didactic paradigms}\label{subsection2.2}

To understand how the paradigm shifts occur in a didactic institution and describe the \textit{\textit{dynamics}} of didactic paradigms, it will be necessary to employ the notion of \textit{\textit{didactic phenomenon}}\footnote{Without attempting to provide a definition of a didactic phenomenon, we shall simply state that it manifests itself through an observable set of events or processes related to study, and that these are classifiable, occur regularly under certain circumstances and across different institutions, and are surprising (from a given perspective), thereby requiring explanation.}. We shall show that didactic paradigms shifts are related with the occurrence or avoidance of certain didactic phenomena.

To interpret the current modality of study of the field $\mathcal{F}$ in $I$, didactic research constructs a DP for the study of $\mathcal{F}$ in $I$, which serves as the reference didactic paradigm, $\text{RDP}_I(\mathcal{F})$, and which models a hypothetical modality of  study of the field $\mathcal{F}$ in $I$:
$$ \text{RDP}_I(\mathcal{F})=\left[\text{REM}_I(\mathcal{F}), \text{RDE}_I(\mathcal{F}), \text{RDM}_I(\mathcal{F})\right].$$

The \textit{economic analysis} of the current modality of study, from the perspective of a $\text{RDP}_I(\mathcal{F})$, provides a model of the aforementioned modality. We refer to this model as \textit{current didactic paradigm (for the study) of} $\mathcal{F}$ \textit{in} $I$, $\text{CDP}_I(\mathcal{F})$, which, as is to be expected, depends on the $\text{RDP}_I(\mathcal{F})$:
$$\text{CDP}_I(\mathcal{F})=\left[\text{CEM}_I(\mathcal{F}),\text{CDE}_I(\mathcal{F}),\text{CDM}_I(\mathcal{F})\right].$$

The choice of a RDP is not arbitrary; it is conditioned by the means sought in the research. In fact, a ${\mathrm{RDP}}_I(\mathcal{F})$ embodies a scientific hypothesis constructed by research that can be formulated as follows: if a study community (meeting certain conditions) were to carry out a study process in $I$ of the field $\mathcal{F}$, represented by $\text{REM}_I(\mathcal{F})$, using the means $\text{RDM}_I(\mathcal{F})$,  then the ends $\text{RDE}_I(\mathcal{F})$ would be achieved, among which the most significant is that of avoiding the current didactic phenomena $\text{CD}\varphi_I(\mathcal{A})$ that occur when the study of the field $\mathcal{F}$ is governed by the $\text{CDP}_I(\mathcal{F})$.

Consequently, the construction of $\text{RDP}_I(\mathcal{F})$ aims at avoiding the phenomena $\text{CD}\varphi_I(\mathcal{F})$, as these are considered ``undesirable''\footnote{It is important to bear in mind: (1) that the consideration of ``undesirable'' didactic phenomena is always made from the perspective of a particular institution (which must be specified in each case); and (2) that not all didactic phenomena under study are necessarily ``undesirable'' from the perspective of didactic research. However, in general, the decision to study a phenomenon is primarily motivated by the need to \textit{explain} it, which already indicates that it contrasts with what would be expected from the standpoint of the assumptions adopted by the research. Moreover, when a phenomenon is chosen for study, it is usually with the aim of acting upon the didactic system in which it appears, so phenomena are typically selected that, from the perspective of the research, produce ``undesirable'' didactic effects, that is, effects \textit{incompatible with the didactic ends} of the DP assumed by the research.} from the research perspective. In short, $\text{RDP}_I(\mathcal{F})$ embodies a possible direction of change for $\text{CDP}_I(\mathcal{F})$, guided by the objective of avoiding the aforementioned phenomena.

The economic analysis of the current modality of study serves as a basis for analysing its ecology, that is, for investigating, first, the conditions that would be required to modify it in the direction indicated by the $\text{RDP}_I(\mathcal{F})$ (prospective analysis). And secondly (retrospective analysis), to answer the following question: how can we explain the emergence and permanence (over a certain period of time) of the current modality of study? We can assume that this emergence constitutes, in some cases, a response to a certain preceding modality of study that we model using a DP that we call the antecedent didactic paradigm, $\text{ADP}_I(\mathcal{F})$:

$$\text{ADP}_I(\mathcal{F})=\left[\text{AEM}_I(\mathcal{F}),\text{ADE}_I(\mathcal{F}),\text{ADM}_I(\mathcal{F})\right].$$

In this case, the current modality of study emerged (and persists) to avoid certain phenomena, $\text{AD}\varphi_I(\mathcal{F})$, which appeared in an antecedent modality of study and were (and continue to be) considered ``undesirable'' from the perspective of $\text{CDP}_I(\mathcal{F})$. However, making this assumption amounts to formulating a scientific hypothesis, which, as such, must be tested in each case, while simultaneously examining what other factors may influence the emergence of a new didactic paradigm. This clarifies why the reference and antecedent paradigms contribute to interpreting the current modality of study: the $\text{RDP}_I(\mathcal{F})$ serves to envisage a possible alternative modality of study, which, by contrast, indicates what the current modality of study is not, while $\text{ADP}_I(\mathcal{F})$ presents it as a response to a possible past modality of study.

Let us note that the pairs $(\text{ADP}_I(\mathcal{F}), \text{CDP}_I(\mathcal{F}))$ and $(\text{CDP}_I(\mathcal{F}), \text{RDP}_I(\mathcal{F}))$ are associated with scientific hypotheses according to which certain didactic phenomena are avoided through the implementation of certain modalities of study. The testing of these hypotheses differs in each case. Indeed, in the case of the hypothesis associated with the retrospective analysis, the researcher may rely on historical documents demonstrating the emergence of $\text{AD}\varphi_I(\mathcal{F})$, and may provide evidence that $\text{AD}\varphi_I(\mathcal{F})$ does not occur in the modality of study governed by $\text{CDP}_I(\mathcal{F})$. In the case of the hypothesis associated with the prospective analysis, the researcher should provide evidence of the existence of the $\text{CD}\varphi_I(\mathcal{F})$ and data supporting that these phenomena would be avoided in a modality of study governed by $\text{RDP}_I(\mathcal{F})$.

In summary, the didactic analysis of a modality of study reveals the dynamics of the didactic paradigms involved and the fundamental role played by didactic phenomena in such dynamics. The didactic analysis begins with the \textit{epistemological analysis} of the knowledge under study, which, in ATD, we call \textit{praxeological analysis} because it is made explicit in terms of praxeologies. Consequently, it will be useful to review the notion of \textit{praxeology}.

\subsection{Praxeologies as instruments for analysing epistemological models}\label{subsection2.3}

To detail the structure of the $\text{EM}_I(\mathcal{F})$ on which the $DP_I(\mathcal{F})$ are based, we shall use the idea of praxeology \citep{praxeologies}. A praxeology is a 4-tuple $\Pi = (T, \tau, \theta, \Theta)$, where:  
\begin{itemize}  
	\item $T$ is a family of task types,  
	\item $\tau$ is a family of techniques (typically described in terms of actions), devoted to carrying out the task types of $T$,  
	\item $\theta$ is the technology, or reasoned discourse on the technique, which explains why the techniques in $\tau$ work, and evaluates their scope, economy, and reliability, and  
	\item $\Theta$ is the theory, which plays the role of a second-level justificatory and interpretative discourse in relation to the knowledge of $\mathcal{F}$ in $I$.  
\end{itemize}

To adequately describe certain didactic phenomena, and even to adequately describe a process of study, we need to specify certain components of $\Theta$. To this end, we shall use the general description of the theory of a praxeology proposed in \citep{gasconciolas2024}, according to which a theory should include, at least:

\begin{itemize}
	\item An \textit{ontological} component, denoted by $\mathcal{O}$, which provides the \textit{language} $L$ used to speak about $\mathcal{F}$, the \textit{interpretation} $Int$ of the non-logical terms of $L$ (that is, terms other than $\lnot, \land, \vee, \rightarrow, \forall, \exists$), and a list of \textit{axioms}. Axiomas are \textit{postulates}, namely elementary statements expressed in terms of $L$ that are assumed to hold (interpreted according to $Int$) without any supporting argument.

	\item A \textit{nomological} component, denoted by $\mathcal{N}$, consisting of \textit{theorems}. Each theorem is a statement expressed in terms of $L$, and, unlike axioms, it is not a starting point in the study of $\mathcal{F}$, but rather the conclusion of a \textit{valid argument} with premises regarded as true according to $Int$ (either because they are themselves theorems or because they are axioms).
	
	\item An \textit{epistemological} component, denoted by $\mathcal{E}$, which establishes what types of arguments are valid to support the theorems.
\end{itemize}

Therefore, in this work we shall consider that the structure of a praxeology can be expressed as  
$\Pi = (T, \tau, \theta, \Theta)$, with $\Theta = (\mathcal{O}, \mathcal{N}, \mathcal{E})$ and $\mathcal{O} = (L, Int, \text{Axioms})$.

\subsection{Simultaneous construction of the current, reference and antecedent paradigm}\label{subsection2.4}

Before initiating the didactic analysis of the current modality of study of $\mathbb{R}$ in DM, it is necessary to clarify some general methodological issues, which we will apply here to the case of the real numbers. As indicated in Section~\ref{subsection2.2}, to interpret the current modality of study of $\mathbb{R}$ in $\text{DM}$, and to represent (model) it through $\text{CDP}_{\text{DM}}(\mathbb{R})$, it is necessary to fix a reference didactic paradigm, $\text{RDP}_{\text{DM}}(\mathbb{R})$, which serves as a reference system. This $\text{RDP}_{\text{DM}}(\mathbb{R})$ models a hypothetical modality of study that is postulated to constitute a way of avoiding certain phenomena, $\text{CD}\varphi _{\text{DM}}(\mathbb{R})$, that arise when $\text{CDP}_{\text{DM}}(\mathbb{R})$ governs the study process. To complete the analysis of the dynamics of $\text{CDP}_{\text{DM} }(\mathbb{R})$, it is necessary to ``explain'' why the aforementioned $\text{CD}\varphi_{\text{DM}}(\mathbb{R})$ occur and where they come from. One possible explanation consists in assuming the existence of an antecedent modality of study, governed by a certain antecedent didactic paradigm, $\text{ADP}_{\text{DM}}(\mathbb{R})$, which is constructed from the perspective of $\text{CDP}_{\text{DM}}(\mathbb{R})$. It is postulated that, when $\text{ADP}_{\text{DM}}(\mathbb{R})$ governed the study, certain phenomena, $\text{AD}\varphi_ {\text{DM}}(\mathbb{R})$, occurred, which were considered undesirable from the perspective of $\text{CDP}_{\text{DM}}(\mathbb{R})$, and that this paradigm was constructed precisely to avoid these phenomena.

We thus see that, on the one hand, in order to rigorously address the necessity of avoiding the phenomena $\text{CD}\varphi_{\text{DM}}(\mathbb{R})$ and to question why they occur and from where they originate, it is necessary to be acquainted with the main features of these phenomena and of the $\text{CDP}_{\text{DM}}(\mathbb{R})$ that governs the modality of study in which they arise. On the other hand, the construction of the $\text{CDP}_{\text{DM}}(\mathbb{R})$ requires, more or less explicitly, the establishment of a $\text{RDP}_{\text{DM}}(\mathbb{R})$. Moreover, the understanding of the $\text{CD}\varphi_{\text{DM}}(\mathbb{R})$ depends, in part, on the answers to why they occur and where they come from, that is, on the construction of an $\text{ADP}_{\text{DM}}(\mathbb{R})$. Ultimately, it turns out that, in scientific practice, the $\text{CDP}$, $\text{RDP}$, and $\text{ADP}$ must be constructed simultaneously, with the construction of each relying on the others. Nevertheless, due to the inevitable linearity of writing, we shall describe them successively, one after another. We shall begin by constructing the $\text{CDP}_{\text{DM}}(\mathbb{R})$ from the perspective of a certain $\text{RDP}_{\text{DM}}(\mathbb{R})$, which we will make explicit in Section~\ref{sec5}. Furthermore, in constructing the $\text{CDP}_{\text{DM}}(\mathbb{R})$, we shall also take into consideration certain features of the transition between the $\text{ADP}_{\text{DM}}(\mathbb{R})$ (which will be described in Section~\ref{sec6}) and the $\text{CDP}_{\text{DM}}(\mathbb{R})$.

\section{Economic analysis: construction of the current didactic paradigm}\label{sec3}

In this section, we shall analyse the economy of the current modality of study of $\mathbb{R}$ in DM from the perspective of a certain $\text{RDP}_{\text{DM}}(\mathbb{R})$. This analysis begins by investigating the functioning of the rules that govern the institutional mathematical organisation concerning $\mathbb{R}$ in $\text{DM}$, that is, by analysing and questioning the current epistemological model of $\mathbb{R}$ in $\text{DM}$. The economic analysis also includes the study of the forms (didactic means and ends) used to organise the study of $\mathbb{R}$ in $\text{DM}$. The results of this analysis will bring to light certain didactic phenomena that will allow the formulation of the research problem (see Section~\ref{sec4}) and that constitute its point of departure.

\subsection{Praxeological analysis of the current epistemological model}\label{subsect3.1}

To characterise the $\text{CEM}_{\text{DM}}(\mathbb{R})$, we began by examining the following textbooks, which are fairly standard and internationally shared: \citep{apostol}, \citep{vina}, \citep{ortega}, \citep{rudin} and \citep{spivak}. We found two ways of introducing $\mathbb{R}$ in $\text{DM}$: either axiomatically (by stating as an axiom that, up to isomorphisms, there exists a unique totally ordered field that satisfies the supremum property), or by defining it from $\mathbb{Q}$ using the well-known Dedekind or Cantor constructions. 

We shall now schematically summarise the construction of the $\text{CEM}_{\text{DM}}(\mathbb{R})$, which aims to reflect certain aspects of what is understood as ``real numbers'' in $\text{DM}$.

\begin{enumerate}
	\item The construction begins with the structure $(\mathbb{N}, +, \cdot, \le)$ in terms of sets. Subsequently, $(\mathbb{Z}, +, \cdot, \le)$ is constructed, where an integer is a certain equivalence class of pairs of natural numbers, and then $(\mathbb{Q}, +, \cdot, \le)$ is constructed, where a rational number is a certain equivalence class of pairs of integers. It is shown that there exist injective mappings $(\mathbb{N}, +, \cdot, \le) \rightarrow (\mathbb{Z}, +, \cdot, \le) \rightarrow (\mathbb{Q}, +, \cdot, \le)$ compatible with addition, multiplication, and order.  All these constructions are presented formally without any explicit reference to interpretation, and the properties of the operations are introduced axiomatically. For example, once $\mathbb{N}$ is constructed and addition has been introduced, multiplication is defined inductively as follows: for any natural numbers $n$ and $m$, $n \cdot 1 := n, \quad n \cdot (m+1) := n \cdot m + n$.  Notice that this definition appeals to the distributive property.
	
	\item It can be shown that $(\mathbb{Q}, +, \cdot, \le)$ is an Archimedean totally ordered field for which some Cauchy sequences are not convergent. This is the case, for example, of the sequence $(p_n/10^n)_{n \in \mathbb{N}}$, where $p_n = \max \{ a \in \mathbb{N} \mid a^2 < 2 \cdot 10^{2n} \}$, or of the sequence $\left(1 + \frac{1}{1!} + \frac{1}{2!} + \ldots + \frac{1}{n!}\right)_{n \in \mathbb{N}}$.
	
	\item The existence of an Archimedean totally ordered field in which every Cauchy sequence is convergent can be prove. This can be done, for example, as proposed by Cantor, by considering the set formed by certain equivalence classes of Cauchy sequences in $\left(\mathbb{Q},+,\cdot,\le\right)$, or by considering the set of Dedekind cuts of $\left(\mathbb{Q},+,\cdot,\le\right)$.
	
	\item The following theorem is proved: ``if two Archimedean and totally ordered fields satisfy that every Cauchy sequence is convergent, then these fields are isomorphic via an isomorphism of ordered fields''.
	
	\item In conclusion, it turns out that, up to an isomorphism of ordered fields, there exists a unique Archimedean totally ordered field in which every Cauchy sequence is convergent. This is called the field of real numbers, $(\mathbb{R}, +, \cdot, \le)$.
	
	\item Finally, it is proved that, excluding periodic decimal expansions ending with a tail of nines, every real number admits a unique decimal representation, and conversely, every decimal expansion represents a unique real number.
\end{enumerate}

When analysing this construction of the $\text{CEM}_{\text{DM}}(\mathbb{R})$ from the perspective of the $\text{REM}_{\text{DM}}(\mathbb{R})$ (which we shall make explicit in detail in Section~\ref{subsect5.1}), a first \textit{praxeological feature} can be observed, which may be described as follows: the operations in $\mathbb{R}$ are based on those in $\mathbb{Q}$, and in turn the operations in $\mathbb{Q}$ are defined axiomatically. Where do these axiomatic definitions, which provide the foundation for the entire construction, come from? Are they arbitrary? Could they have been different?

A brief analysis of these definitions shows that, far from being arbitrary, they are based on the interpretation of rational numbers as \textit{exact measures of quantities of magnitude commensurable} with the unit of measurement, or, equivalently, as \textit{solutions of linear equations}. This observation highlights a paradox in the construction of the $\text{CEM}_{\text{DM}}(\mathbb{R})$: on the one hand, it disregards the relationship between real numbers and the measurement of quantities of magnitude, but, on the other hand, implicitly, it is grounded in this relationship. This \textit{separation between real numbers and the measurement of magnitudes} originates in the near disappearance of the problem of measurement in the school curriculum. 

A \textit{second praxeological} feature visible in the $\text{CEM}_{\text{DM}}(\mathbb{R})$ is the \textit{``artificial'' nature of the constructions} of $\mathbb{R}$. These constructions, whether they concern the axiomatic definition or the well-known Dedekind and Cantor constructions, conceal their raison d'\^etre by always avoiding the use of geometric intuitions in the arguments.

This avoidance constitutes an essential aspect of the \textit{epistemological component} of the praxeologies that form part of the $\text{CEM}_{\text{DM}}(\mathbb{R})$ and, as we shall see in Section~\ref{sec6}, represents a radical shift compared to the antecedent paradigm, which was grounded in geometric intuition. In particular, this shift affects the meaning of operations between real numbers. When these numbers were quantities of magnitude, there was no doubt that certain operations could be performed with them. For example, the product $\sqrt{2} \cdot \sqrt{3} = \sqrt{6}$ could be carried out because, for instance, it was clear that a rectangle with dimensions $\sqrt{2}$ and $\sqrt{3}$ must have an area, and it was also clear that $$(\sqrt{2} \cdot \sqrt{3})^2 = (\sqrt{2} \cdot \sqrt{3}) \cdot (\sqrt{2} \cdot \sqrt{3}) = \sqrt{2} \cdot \sqrt{2} \cdot \sqrt{3} \cdot \sqrt{3} = 2 \cdot 3 = 6.$$

However, with the new constructions of the real numbers, how can one even prove that the operation $\sqrt{2} \cdot \sqrt{3}$ can be carried out? These are also the types of tasks to which space is devoted in the textbooks of the new $\text{CDP}_{\text{DM}}(\mathbb{R})$. Dedekind  himself, referring to the definition of operations in $\mathbb{R}$, writes:

\begin{quote}
	Just as addition is defined, so can the other operations of the so-called elementary arithmetic be defined, $\ldots$ and in this way we arrive at real proofs of theorems (as, e.g., $\sqrt2\cdot\sqrt3=\sqrt6$), which to the best of my knowledge have never been established before \citep[p. 11]{dedekind1963}.
\end{quote}

This shift in the \textit{epistemological component} had consequences for the \textit{ontological component}, not only regarding to the ontology of the real numbers (what kind of objects are real numbers?), but also for other notions in mathematical analysis, which came to be interpreted without reference to magnitudes. Thanks especially to the work of Weierstrass, it became possible to adopt a purely arithmetical-algebraic language that allowed the basic notions of analysis to be presented in terms of operations and inequalities (in terms of $\varepsilon$-$\delta$). Thus, we encounter a $\text{CEM}_{\text{DM}}(\mathbb{R})$ that is entirely different from its predecessor. These differences are evident when comparing nineteenth-century textbooks with those of the early twentieth century. For example, \cite{kline1990} indicates that it is possible to confirm the change of epistemological model in the study of mathematical analysis (and in particular of the real numbers) by comparing the first with the second and third editions of Camille Jordan's \textit{Cours d'analyse}. By contrasting the first volume of the first edition \citep{jordan1882} with that of the second edition \citep{jordan1893}\footnote{It is possible to consult the digitised versions of the three editions of Jordan's \textit{Cours d'analyse} in the digital library of the French National Library.}, we can confirm that the differences are indeed remarkable. The treatment of the real numbers in the first edition still relies essentially on geometric notions, and although definitions of continuity and differentiability comparable to modern ones are introduced, a rigorous and deductive study of their properties is not presented. In the second edition, Dedekind cuts are introduced in the first chapter, with significant consequences for the development of the notions of limit, continuity, and differentiability: fundamental properties of the real numbers are presented and proven (for example, in the first chapter, it is shown, without appealing to geometric intuition, that every bounded monotone increasing sequence of real numbers is convergent).

A similar process occurred in the introduction of modern analysis by Zoel Garc\'ia de Galdeano\footnote{The digitised versions of Garc\'ia de Galdeano's works can be accessed in the digital library of the Spanish National Library.}  in Spain between the late nineteenth and early twentieth centuries \citep{ausejo2010}. In his \textit{Tratado de \'Algebra} of 1884, Garc\'ia de Galdeano introduced some analytical notions based on geometric intuition, while in his \textit{Tratado de An\'alisis Matem\'atico} of 1904, the arithmetization of analysis was presented through the introduction of Dedekind cuts.

\subsection{Current didactic ends, means and phenomena}\label{subsect3.2}

The two praxeological features identified in the $\text{CEM}_{\text{DM}}(\mathbb{R})$ (the separation between numbers and the measurement of magnitudes, on the one hand, and the artificial character of the constructions of the real numbers and the consequent obscuring of their  raison d'\^etree, on the other) together with the new epistemological and ontological components, give rise to didactic phenomena that are ``undesirable'' from the perspective of the ATD.

Indeed, if the study of the real numbers in DM is governed by the $\text{CDP}_{\text{DM}}(\mathbb{R})$, then the didactic means exhibit a tendency towards authoritarianism: students are expected to possess sufficient ``mathematical maturity'' to accept deductive arguments without any ``unnatural'' concern for the problematic background that lies behind them and that would allow the aforementioned constructions of the real numbers to be made meaningful. This authoritarianism of the didactic means is not confined to the study of real numbers in DM; it is in fact a phenomenon that manifests itself at the disciplinary level \citep{lakatos1978b}. The new ``authoritarian'' epistemological component (only axiomatic definitions and deductive arguments whose validity does not depend on our intuitions about certain magnitudes), together with the corresponding ontological component, requires explicit presentation. Thus, for example, in the Degree in Mathematics at the University of Murcia, a course entitled \textit{Introducci\'on al m\'etodo matem\'atico} has been included, in which, among other topics, the distinction between natural and formal language, first-order logic, common strategies for proof in mathematics, notions concerning the axiomatic development of mathematics, models, paradoxes, and the limitations and variations of axiomatics are studied. The existence of this course is explained by the fact that the epistemological model of mathematics in Secondary Education is very distant from the $\text{CEM}_{\text{DM}}(\mathbb{R})$, and highlights a neglect of the mathematical problematic in favour of focusing on the ``form'' in which responses to that problematic are presented.

On the other hand, the study of $\mathbb{R}$ aims to ensure that students acquire the necessary foundations to avoid the difficulties encountered by geometric intuition and, above all, to advance in their future studies, to support rigorously the notions of limit, continuity, differentiation, and integration \citep{edwards1979}. This objective leads to a neglect of the role played by $\mathbb{R}$ in the mathematical tasks that students perform ``currently'' \citep{licera2019}. For this reason, we say that the \textit{didactic ends} of the study of $\mathbb{R}$ in DM \textit{are of a propaedeutic nature}, which constitutes another of the current didactic phenomena.

It is worth noting that in the current modality of study the aforementioned difficulties arising from the inadequate use of geometric intuition are avoided, but not explicitly mentioned, so that the question of why real numbers are constructed in such an artificial manner remains unanswered. In fact, the constructions are clearly motivated by their usefulness in proving the existence of certain suprema (the construction of Dedekind cuts) and the convergence of certain sequences of rational numbers (Cantor's construction, which presents the real numbers as equivalence classes of Cauchy sequences over $\mathbb{Q}$). Consistently, many types of tasks in the $\text{CDP}_{\text{DM}}(\mathbb{R})$ are directly concerned with these matters. Examples of task types would be:

\begin{itemize}  
	\item proving that certain subsets of $\mathbb{Q}$, although bounded above, have no supremum in $\mathbb{Q}$. For example, $\{x \in \mathbb{Q} \mid x > 0, \, x^2 < 2\}$;  
	\item proving that the sequence $\left\{\tfrac{p_n}{10^n}\right\}_{n \in \mathbb{N}}$, where $\tfrac{p_n}{10^n}$ is the largest rational number with denominator $10^n$ whose square is less than $2$, is Cauchy but has no limit in $\mathbb{Q}$;  
	\item proving that if $A \subseteq \mathbb{Q}$ has a supremum, then there exists a sequence $(a_n)_{n \in \mathbb{N}}$ with $a_n \in A$ and $\lim_{n \to \infty} a_n = \sup A$;  
	\item proving that if $(a_n)_{n \in \mathbb{N}}$ is a bounded above and monotone increasing sequence of rational numbers, then it is convergent and its limit is $l = \sup \{a_n\}_{n \in \mathbb{N}}$;  
	\item proving that the sequence with general term $a_n = 1 + 1/1 + 1/2 + \ldots + 1/n$ satisfies that, given $\varepsilon > 0$, there exists $n_0 \in \mathbb{N}$ such that $\lvert a_{n+1} - a_n \rvert < \varepsilon$ for $n > n_0$, but that it is nevertheless not a Cauchy sequence\footnote{The sequence $\left(b_n\right)_n$ with $b_n=a_n-\log{\left(n\right)}=\sum_{m=1}^{n}\frac{1}{m}-\log{\left(n\right)}$ is a Cauchy sequence (so convergent). Curiously enough, it is still not known whether its limit $\gamma$ (the so-called \textit{Euler-Mascheroni constant}) is rational or irrational.}.  
\end{itemize}

Yet it remains unclear why it is convenient to have a set of numbers in which every non-empty bounded-above set has a supremum, or in which the convergent sequences are precisely the Cauchy sequences.

The aforementioned current phenomena are, to a large extent, a consequence of the \textit{arithmetization of analysis}, which in turn is closely related to the formalisation of mathematics, consisting in the exclusive use of axiomatic definitions and purely deductive arguments that avoid relying on intuition about the behaviour of the objects involved.

\section{Research problem}\label{sec4}

The didactic phenomena described above (\textit{the authoritarian nature of the didactic means} and \textit{the propaedeutic character of the didactic ends}) are interrelated and regarded as undesirable from the perspective of the ATD. They constitute the starting point of our research. In order to explain them, and eventually to avoid them, we pose the following \textit{research problem}, which we formulate in the form of questions:

Why are the current didactic means for the study of $\mathbb{R}$ in DM authoritarian, and the current didactic ends are essentially propaedeutic in nature? Can it be justified that these $\text{CD}\varphi_{\text{DM}}(\mathbb{R})$ occur as an unforeseen and unintended consequence (a latent function) of the implementation of the $\text{CDP}_{\text{DM}}(\mathbb{R})$? How can a new modality of study of $\mathbb{R}$ in DM, governed by a paradigm consistent with that of mathematical modelling, be designed, implemented, and evaluated so as to avoid these $\text{CD}\varphi_{\text{DM}}(\mathbb{R})$? What conditions would be required and what constraints might hinder the transition from the $\text{CDP}_{\text{DM}}(\mathbb{R})$ towards this new paradigm, which we call $\text{RDP}_{\text{DM}}(\mathbb{R})$? Is it possible to determine the existence of an antecedent didactic paradigm, $\text{ADP}_{\text{DM}}(\mathbb{R})$, and antecedent didactic phenomena, $\text{AD}\varphi_{\text{DM}}(\mathbb{R})$, regarded as undesirable from the perspective of the $\text{CDP}_{\text{DM}}(\mathbb{R})$, which would allow the historical emergence of the $\text{CDP}_{\text{DM}}(\mathbb{R})$ to be interpreted as a way of avoiding such phenomena (a manifest function)\footnote{The avoidance of the $\text{AD}\varphi_{\text{DM}}(\mathbb{R})$ could be regarded as part of the manifest functions, that is, the objective and intended consequences, of the implementation of the $\text{CDP}_{\text{DM}}(\mathbb{R})$, whereas the occurrence of the $\text{CD}\varphi_{\text{DM}}(\mathbb{R})$ could be considered as latent functions, that is, the unintended and unforeseen consequences, of the implementation of the $\text{CDP}_{\text{DM}}(\mathbb{R})$. The distinction between manifest and latent functions was studied in detail by \cite{merton1949}.}? What conditions have favoured, and what constraints have hindered, the historical transition from the $\text{ADP}_{\text{DM}}(\mathbb{R})$ to the $\text{CDP}_{\text{DM}}(\mathbb{R})$?

To address these questions, we shall carry out an ecological analysis (both retrospective and prospective) of the current modality of study of $\mathbb{R}$ in DM. This is a diachronic analysis insofar as it integrates: (a) questions concerning the conditions that made possible the emergence of the modality of study of $\mathbb{R}$ in DM and that sustain its current validity; and (b) questions concerning the conditions required (and the constraints that hinder) possible changes in this modality of study in a given direction. 

We postulate that the results of this analysis, based on the economic analysis described in Section~\ref{sec3} and, in particular, on the praxeological analysis (see Section~\ref{subsect3.1}), will provide elements of response to the questions that form part of the research problem.

\section{Prospective analysis: construction of a reference didactic paradigm}\label{sec5}

The authoritarian nature of the didactic means and the essentially propaedeutic character of the didactic ends, visible when the study of $\mathbb{R}$ in DM is governed by the $\text{CDP}_{\text{DM}}(\mathbb{R})$, constitute the current didactic phenomena, $\text{CD}\varphi_{\text{DM}}(\mathbb{R})$. In this section, we describe a possible modality of study of $\mathbb{R}$  in DM which, we postulate, would allow these phenomena to be avoided. To this end, we shall construct a $\text{RDP}_{\text{DM}}(\mathbb{R})\ =\ [\text{REM}_{\text{DM}}(\mathbb{R}),\ \text{RDE}_{\text{DM}}(\mathbb{R}),\ \text{RDM}_{\text{DM}}(\mathbb{R})]$ that governs this modality of study and is based on the $\text{REM}_{\text{DM}}(\mathbb{R})$ that we have implicitly used to analyse the $\text{CEM}_{\text{DM}}(\mathbb{R})$ in Section~\ref{subsect3.1}, and which we shall construct explicitly below.

\subsection{Reference epistemological model}\label{subsect5.1}

To achieve the $\text{RDE}_{\text{DM}}(\mathbb{R})$, which will be made explicit in Section\ref{subsect5.2}, we require an epistemological model alternative to the current one, which begins by connecting real numbers with the measurement of magnitudes, but ultimately presents them in a way compatible with the formalist epistemological component of the $\text{CEM}_{\text{DM}}(\mathbb{R})$, according to which only axiomatic definitions and deductive arguments are used (although in our epistemological model these definitions are motivated by considerations about magnitudes).

To understand our $\text{REM}_{\text{DM}}(\mathbb{R})$, we will also outline the epistemological models from which it inherits, and which it presupposes, namely: $\text{REM}_{\text{DM}}(\mathbb{M})$, $\text{REM}_{\text{DM}}(\mathbb{N})$, $\text{REM}_{\text{DM}}(\mathbb{Q}^+)$ and $\text{REM}_{\text{DM}}(\mathbb{R}^+)$, where $\mathbb{M}$ denotes the study of magnitudes, and $\mathbb{N}$, $\mathbb{Q}^+$, and $\mathbb{R}^+$ denote the study of the corresponding metrizations (using non-negative numbers) of certain magnitudes.

In the $\text{REM}_{\text{DM}}(\mathbb{M})$ we begin by observing that a magnitude can be modelled through an \textit{extensive comparative system}\footnote{In \citep{diez2008} this is called a \textit{metric}, but we follow the terminology of \citep{mosterin2016}.}, which incorporates into the ontological component a non-empty set $A$, together with binary relations, $\sim$ and $\prec$, on $A$, and a binary law $\circ$ of internal composition on $A$, satisfying the following axioms:

\begin{itemize}
	\item $\sim$ is an equivalence relation,
	\item For all $x,y,z \in A$, $(x \prec y) \wedge (y \prec z) \Rightarrow (x \prec z)$,
	\item For all $x,y \in A$, $(x \prec y) \Rightarrow \lnot (y \prec x)$,
	\item For all $x,y \in A$, $(x \prec y) \vee (y \prec x) \vee (x \sim y)$,
	\item For all $x,y,z \in A$, $x \circ (y \circ z) \sim (x \circ y) \circ z$,
	\item For all $x,y \in A$, $x \circ y \sim y \circ x$,
	\item For all $x,y,z \in A$, $x \prec y \Leftrightarrow x \circ z \prec y \circ z$,
	\item For all $x,y \in A$, $x \neq y \Rightarrow x \prec x \circ y$,
	\item For all $x,y \in A$, there exists a natural number $n$ such that $x \prec n y$ where $$n y :=\overbrace{y \circ \cdots \circ y}^{n\text{ times}}.$$
\end{itemize}

Our interpretation is as follows: the elements of $A$ are the objects being compared with respect to a magnitude. The interpretation we give to the statements $x \prec y$ and $x \sim y$ depends on a\textit{ standard technique of direct comparison} associated with the magnitude in question, and $x \circ y$ is an object resulting from combining (in a manner depending on the particular magnitude) the objects $x$ and $y$. For a more precise interpretation, a specific magnitude must be fixed. For instance, if the magnitude is cardinality, the elements of $A$ would be finite sets, the comparison would be carried out using the technique of direct pairing (matching elements of $x$ with elements of $y$), and the combination of two sets would be performed via the union of two disjoint sets. If the magnitude is length, the elements of $A$ would be segments (idealized objects), the comparison would be carried out using the technique of side-by-side juxtaposition (placing one segment next to the other, aligning one end to see if either extends beyond), and the combination of two segments would be done via concatenation. If the magnitude is area, the elements of $A$ would be surfaces (idealized objects), the comparison would be carried out using the technique of superposition (to see whether one surface fits within the other), and the combination via juxtaposition of surfaces.

In $\text{REM}_{\text{DM}}(\mathbb{N})$, we begin by observing that, in practice, it is often not possible to apply the technique of direct comparison, so an \textit{indirect comparison} must be made. In the case of the magnitude of cardinality, this leads to the creation of the natural numbers \citep{sierra2006,garcia2015}. 

We then incorporate into our ontological component the familiar objects and relations $(\mathbb{N}, =, <, +)$, together with a mapping $\mu: A \rightarrow \mathbb{N}$ such that:
\begin{align}
	x \sim y &\Leftrightarrow \mu(x) = \mu(y), \label{propmed1}\\
	x \prec y &\Leftrightarrow \mu(x) < \mu(y), \label{propmed2}\\
	\mu(x \circ y) &= \mu(x) + \mu(y). \label{propmed3}
\end{align}

We call \emph{measurement} the process of assigning a number $\mu(x)$ to a set $x$, and, consistently, we call $\mu$ a \emph{measure mapping}. Observe that properties \eqref{propmed1} and \eqref{propmed2} of the measure mapping allow us to compare indirectly: if we want to compare $x$ and $y$, we do not need to apply the standard comparison technique; it suffices to measure them, obtaining the corresponding numbers $\mu(x)$ and $\mu(y)$, compare these numbers, and then conclude whether $x \prec y$, $y \prec x$, or $x \sim y$. 

The justification that $\text{REM}_{\text{DM}}(\mathbb{N})$ provides for the natural numbers, as tools for comparing finite sets with respect to the magnitude of cardinality, supports the usual construction of natural numbers in Set Theory: $0 := \emptyset$, $1 := \{\emptyset\}$, $2 := \{\emptyset,\{\emptyset\}\}, \ldots$

Sometimes, measuring a set directly is not easy, and we must perform an \textit{indirect measurement}. For this purpose, the \textit{operations} are useful. For example, \textit{addition} is used to measure a set that is given as the disjoint union of two sets with known measures. \textit{Multiplication} is used to measure a set that is given as the disjoint union of several equipotent sets with known measures. The interpretation of these operations justifies (through non-deductive arguments) the associative and commutative properties for addition and multiplication, as well as the distributive property. This provides a justification for the usual definition of the sum and product of natural numbers in Set Theory:
\begin{itemize}
	\item Addition: $m+0:=m$,\ \ $m+\left(n+1\right):=\left(m+n\right)+1$.
	\item Multiplication: $m\cdot1:=m$,\ \ $m\cdot\left(n+1\right)=m\cdot n+m$.
\end{itemize}

In $\text{REM}_{\text{DM}}(\mathbb{Q}^+)$, it is also frequently impossible, in practice, to apply the technique of direct comparison for magnitudes such as length or area, which leads to the creation of the positive rational numbers to enable indirect comparison. Natural numbers are not sufficient because, in some cases, an object is not equivalent to an integer multiple of the unit of measure, but rather to a certain number of times the \textit{nth} part of the unit of measure (or, equivalently, the \textit{nth} part of a certain number of repetitions of the unit of measure). With this interpretation, one can provide a (non-deductive) argument showing that two fractions of the unit of measure, $\frac{a}{b}$ and $\frac{c}{d}$, represent the same magnitude if and only if $a\cdot d=b\cdot c$. This justifies the usual construction of the rational numbers in Set Theory. We then incorporate into our ontological component the well-known objects and relations, $\left(\mathbb{Q}^+,=,<,+\right)$, and a mapping $\mu:A\rightarrow\mathbb{Q}^+$ that is intended to satisfy properties \eqref{propmed1}, \eqref{propmed2} and \eqref{propmed3} above, so as to allow \textit{indirect comparison}, just as occurred with the natural numbers and the magnitude of cardinality.

As was the case in $\text{REM}_{\text{DM}}\ (\mathbb{N})$, direct measurement is sometimes difficult, and it becomes necessary to use the operations to perform indirect measurement. For example, addition would serve to measure a segment given as the concatenation of two segments of known length. Multiplication would serve to measure the area of a rectangle with sides of known lengths, or to measure a length or area given as a fraction of a fraction of a given length or area. With these interpretations of the operations, the usual definitions of addition and multiplication in Set Theory are justified (through non-deductive arguments).

However, unlike what happens with the natural numbers, it could initially be questioned whether, in some cases, when measuring the length of a segment or the area of a surface, rational numbers provide only approximate measures rather than the exact measure. In such a case, $\mu$ would be defined only up to a certain degree of precision and would satisfy properties \eqref{propmed1}-\eqref{propmed3} only within that degree of precision. Indeed, given a unit of measurement u and an object $x\in A$, it could happen that there exist no natural numbers $a,b\in\mathbb{N}$ such that $b\cdot x\sim a\cdot u$.

In $\text{REM}_{\text{DM}}(\mathbb{Q}^+)$, it will be seen that the non-negative rational numbers correspond exactly to the non-negative periodic decimal expansions. Using the Pythagorean theorem, it can be shown that the diagonal of a square with side $1$ is a number whose square is $2$, and there is no rational number or non-negative periodic decimal expansion whose square is $2$ (see \citep{klazar2009}, where this fact is proved using only decimal expansions). But what if, instead of using only non-negative periodic decimal expansions, we use all non-negative decimal expansions?

We now begin the outline of $\text{REM}_{\text{DM}}(\mathbb{R}^+)$, related to \citep{licera2017}. Let $\mathbb{R}^+$ be the set of all non-negative decimal expansions, which can be viewed as pairs $(a_0,(a_n)_{n\geq1})$, where $a_0$ is a natural number and $a_n\in\mathcal{D}:=\{0,1,\ldots,9\}$. The measurement function is then defined as follows. We fix a segment, $u_0$, which is taken as the unit of measure, and for each natural number $n\geq1$ we take a segment $u_n$ such that ${10}^n\cdot u_n=u_0$. Then the measurement function $\mu:A\rightarrow\mathbb{R}^+$, $x\mapsto\mu\left(x\right)=a_0.a_1a_2a_3\ldots$, is computed as follows:
\begin{itemize}
	\item $a_0$ is the greatest natural number such that $\lnot\left(x \prec a_0u\right)$,
	\item $a_1$ is the greatest number at $\mathcal{D}={0,1,\ldots,9}$ such that $\lnot\left(x\prec a_0u\circ a_1u_1\right)$,
	\item $a_2$ is the greatest number at $\mathcal{D}$ such that $\lnot\left(x\prec a_0u\circ a_1u_1\circ a_2u_2\right)$,
	\item $\ldots$
\end{itemize}

It seems clear that by means of this process, and after a finite or countable number of steps, we would obtain a decimal expansion (periodic or not) that provides the measure of any segment, since it approximates its length beyond any desired degree of precision. For instance, by applying this procedure we see that the first digits of the measure of the diagonal of a unit square would be $1.4142\ldots$ Indeed, $1$ is the greatest natural number whose square does not exceed $2$; $1.41$ is the greatest decimal expansion with only one decimal digit whose square does not exceed $2$; $1.41$ is the greatest decimal expansion with two decimal digits whose square does not exceed $2$; and so on. In this way, we obtain an increasing sequence bounded above, $1,\ 1.4,\ 1.41,\ 1.414,\ 1.4142,\ \ldots$, which can be seen as a series of truncations of a decimal expansion whose full set of digits cannot be determined in a finite number of steps, but of which any given digit can be determined in finitely many steps. The square of the corresponding non-periodic decimal should then be $2$. But is it really so? More generally, can we operate with non-periodic decimal expansions?

Given two decimal expansions, we can construct sequences of finite decimal expansions by operating with their truncated forms of increasing length. In this way, as in the previous example, we obtain an increasing and bounded sequence of finite decimals, which determines an element of $\mathbb{R}^+$. To simplify notation, let $a = (a_0,(a_n)_{n\ \geq\ 1})$ be a decimal expansion. For each $m\ \geq\ 1$, we denote by $a_{(m)}$ the decimal expansion obtained by truncating $a$ at the $m$-th decimal place, that is, $a_{(m)}\ =\ (a_0,\ a_1,\ a_2,\ \ldots,\ a_m,\ 0,\ 0,\ \ldots)$.This decimal expansion is, of course, identified with the rational number $\sum_{n=0}^m\ \frac{a_n}{10^n}$.

It can be shown that $\mathbb{R}^+$, equipped with the usual lexicographic order\footnote{We shall assume that decimal expansions ending in a tail of nines are identified, in the usual way, with a finite decimal expansion; for example, $1.4999\ldots = 1.5$. Alternatively, one could begin the whole construction by excluding decimal expansions that end in a tail of nines and then later adjoin them and prove equalities of the kind above.}, satisfies the supremum property; that is, for every non-empty subset $S\subseteq\mathbb{R}^+$ that is bounded above there exists $\alpha=\sup(S)\in\mathbb{R}^+$ which is the least upper bound of $S$. It is important to observe that seeking completeness (expressed by the supremum property) now appears natural, since it gives a mathematical reinforcement of the idea that, even when a countably infinite number of steps is required, the measuring process terminates in the sense that it determines uniquely a decimal expansion. We shall now show how to prove the supremum property in $\mathbb{R}^+$.

\begin{theorem}[Supremum property] Every non-empty and bounded above subset of $\mathbb{R}^+$ has a least upper bound in $\mathbb{R}^+$.
\end{theorem}
\begin{proof}
Let $S\subseteq\mathbb{R}^+$ be non-empty and bounded. We shall show there exists some $\alpha=\left(a_0,(a_n)_{n\geq1}\right)\in\mathbb{R}^+$ which is the supremum of $S$. 

If $S$ has a maximum, there is nothing to prove (that maximum is already the supremum). It what follows, let us then suppose  that $S$ has no maximum, i.e., that there is no element at $S$ being an upper bound of $S$. 

We begin taking $a_0$ as the greatest natural number which is not an upper bound of $S$ (its existence is guarantee by the well-ordering principle of $\mathbb{N}$).
The sequence of digits $(a_n)_{n\geq1}$ is defined recursively as follows:
\begin{itemize}
	\item $a_1$ is the greatest digit $\left(a_1\in\mathcal{D}\right)$ such that $a_0+a_1/10$ is not an upper bound of S,
	\item assuming the first m digits $a_1,\ldots,a_m$ have been defined, $a_{m+1}$ is chosen as the greatest digit $a_{m+1}\in\mathcal{D}$ such that $a_0+\sum_{n=1}^{m+1}\frac{a_n}{{10}^n}$ is not an upper bound of $S$.
\end{itemize}

We shall prove that $\alpha=\left(a_0,(a_n)_{n\geq1}\right)$ is the supremum of $S$.

We begin showing that $\alpha$ is an upper bound of S. To do so, let us take $(b_0,(b_n)_{n\geq1})\in S$ and show that $(b_0,(b_n)_{n\geq1})\le\alpha$. Since $(b_0,(b_n)_{n\geq1})$ is not a maximum of $S$ (recall we are assuming $S$ has no maximum), $b_0$ cannot be an upper bound of $S$, so $b_0\le a_0$ (by definition of $a_0$). If $b_0<a_0$, then $(b_0,(b_n)_{n\geq1})\le(a_0,{(a}_n)_{n\geq1})=\alpha$ and we have done. Otherwise, if $a_0=b_0$, we have that $a_0+b_1/10$ is not an upper bound of $S$ because $S$ has no maximum and $a_0+b_1/10\le(b_0,(b_n)_{n\geq1})$. Hence, by definition of $a_1$, we must have $b_1\le a_1$. Proceeding recursively we arrive at one of the two mutually exclusive possibilities:

\begin{itemize}
	\item $b_n=a_n$ for every $n\in\mathbb{N}$, that is, $(b_0,(b_n)_{n\geq1})=(a_0,(a_n)_{n\geq1})$, or
	\item there exists $m\in\mathbb{N}$ such that $b_n=a_n$ for every $0\le n<m$ and $b_m<a_m$, that is, $(b_0,{b_n}_{n\geq1})<(a_0,{a_n}_{n\geq1})=\alpha$.
	\end{itemize}
	
It remains to show that $\alpha$ is the least upper bound of $S$. Indeed, if $\beta\in\mathbb{R}^+$ is such that $\beta<\alpha$, then there exists $m\in\mathbb{N}$ such that $\beta<a_0+\ \sum_{n=1}^{m}{a_n/{10}^n}$ and the definition of $a_m$ guarantees that $\beta$ cannot be an upper bound of $S$.
\end{proof}

Once the supremum property has been established for $\mathbb{R}^+$, addition and multiplication are defined as follows:

\begin{itemize}
	\item Given two elements $a=(a_0,(a_n)_{n\ge 1})$ and $b=(b_0,(b_n)_{n\ge 1})$ of $\mathbb{R}^+$, their sum is defined by $a+b := \sup_{m\ge 1} \{a_{(m)}+b_{(m)}\}$, where $a_{(m)}$ and $b_{(m)}$ denote the truncations of $a$ and $b$ at the $m$-th decimal place.
	\item Similarly, their product is defined by $a\cdot b := \sup_{m\ge 1}\{a_{(m)}\cdot b_{(m)}\}$.
\end{itemize}

It can be shown that these operations admit additive and multiplicative identities, and satisfy the associative, commutative and distributive laws. Moreover, they are compatible with the order on $\mathbb{R}^+$, and every non-zero decimal expansion has a multiplicative inverse\footnote{It is also possible to establish effective algorithms for the addition and multiplication of decimal expansions, without appealing to the computation of suprema \citep{fardin2021}.}.

The work with elementary magnitudes can also justify the convenience of introducing negative numbers. Indeed, we may want to speak of negative quantities of length in order to measure the distance from a point in a certain direction (for example, to the left of a fixed point on a line). Similarly, we may want to speak of negative quantities of time, to measure the lapse that has passed before a certain predetermined instant. These quantities would be like the positive ones, but with a sign denoting their peculiar interpretation. Thus, for each non-zero $x\in\mathbb{R}^+$ we add an element $x^\ast$ such that $x+x^\ast=0$. If we call $\mathbb{R}$ the set resulting from adding to $\mathbb{R}^+$ the corresponding elements $x^\ast$, and we want to extend addition and multiplication to $\mathbb{R}$ while preserving the associative, commutative, and distributive properties, we obtain the rule of signs, namely:
$x\cdot y^\ast:=(x\cdot y)^\ast,\quad x^\ast\cdot y^\ast:=x\cdot y$. We can also deduce that addition must satisfy
\[
x + y^\ast =
\begin{cases}
	x - y, & \text{if } y < x,\\
	0, & \text{if } y = x,\\
	(y - x)^\ast, & \text{if } y > x,
\end{cases}
\]
for $x, y \in \mathbb{R}^+$, where $y^\ast$ denotes the additive inverse of $y$.

Another option is to add the negative numbers immediately after the natural numbers, and work with both positive and negative numbers from that point onward. Thus, instead of the sequence
\[
\text{REM}_{\text{DM}}(\mathbb{N}) \rightarrow \text{REM}_{\text{DM}}(\mathbb{Q}^+) \rightarrow \text{REM}_{\text{DM}}(\mathbb{R}^+) \rightarrow \text{REM}_{\text{DM}}(\mathbb{R}),
\]
we would have the following sequence:
\[
\text{REM}_{\text{DM}}(\mathbb{N}) \rightarrow \text{REM}_{\text{DM}}(\mathbb{Z}) \rightarrow \text{REM}_{\text{DM}}(\mathbb{Q}) \rightarrow \text{REM}_{\text{DM}}(\mathbb{R}).
\]

In this case, the integers can be formalized in several ways within Set Theory. A common construction is to identify an integer with the equivalence class $\left[(a,b)\right]$ of a pair $(a,b)$ of natural numbers, where this class represents the integer $a-b$, and two pairs $(a,b)$ and $(c,d)$ are equivalent if $a+d = b+c$. In this case, the operations are defined as:
\begin{eqnarray*}
	\left[(a,b)\right] + \left[(c,d)\right] &:=& \left[(a+c, b+d)\right],\\
	\left[(a,b)\right] \cdot \left[(c,d)\right] &:=& \left[(ac+bd, ad+bc)\right].
\end{eqnarray*}

Note that the definition of the product follows from the sign rule, which can be deduced from the desire to extend the addition and multiplication operations to a structure that consistently includes additive inverses. The definition of rational numbers is done similarly to the case of $\mathbb{Z}$. The definition of $\mathbb{R}$ is done as in $\mathbb{R}^+$, but allowing the first component of the ordered pair to be an integer rather than a natural number.

\subsection{Reference didactic ends and means}\label{subsect5.2}

Among the $\text{RDE}_{\text{DM}}(\mathbb{R})$ is the avoidance of the $\text{CD}\varphi_{\text{DM}}(\mathbb{R})$. As we have indicated, these phenomena emanate from an epistemological model, the $\text{CEM}_{\text{DM}}(\mathbb{R})$, which only admits axiomatic definitions and deductive arguments, leaving aside modelling activity and, in particular, work with magnitudes. It is for this reason that we have constructed a $\text{REM}_{\text{DM}}(\mathbb{R})$ that identifies work with numbers as the result of a modelling activity of quantities of magnitude. Moreover, this is compatible with a general didactic paradigm for the study of mathematics which, in other works \citep{gascon2021a, gascon2021b, barquero2013, garcia2006modelling}, we have called \textit{the mathematical modelling paradigm}, and which we postulate will prevent other phenomena related to the general study of mathematics in different didactic institutions. According to this disciplinary paradigm, all mathematical praxeologies can be seen as the result of modelling processes of (extra- or intra-mathematical) systems, whose ultimate purpose is to be able to answer questions about these systems.

In contrast to what occurs with the current epistemological model, which seems compatible only with authoritarian didactic means, the reference epistemological model is, of course, compatible with lectures, but also with student inquiry activities guided by the teacher, for example through the so-called \textit{Study and Research Paths} \citep{barquero2013}. This is possible because the praxeological elements emerge naturally, always motivated by questions arising throughout our study of the comparative systems associated with certain magnitudes. It should be noted that this is only possible if the epistemological component admits non-deductive arguments, which are unavoidable in the construction of a model of a system.

\section{Retrospective analysis: construction of an antecedent didactic paradigm}\label{sec6}

The retrospective analysis arises from the attempt to understand why the study of $\mathbb{R}$ in DM has come to be governed by the $\text{CDP}_{\text{DM}}(\mathbb{R})$. What was the antecedent modality of study that makes it possible to explain the emergence of the current one?

With the aim of attempting to address this question, we shall construct, from the perspective of the $\text{CDP}_{\text{DM}}(\mathbb{R})$, an $\text{ADP}_{\text{DM}}(\mathbb{R})$ that embodies a possible past of the current modality of study of $\mathbb{R}$ in DM. Strictly speaking, if we intend to speak of a possible past modality of study, we should take into account that the DM institution has evolved enormously over time. By a slight abuse of notation, we shall continue denoting by DM the didactic institution that, in the past, was responsible for training prospective researchers in mathematics. And, again strictly speaking, it would not be accurate to speak of the study of real numbers in the institution responsible for the education of mathematicians. Indeed, until the nineteenth century, the mathematical community made merely pragmatic use of the notion of irrational number, without critically examining its precise meaning or nature\footnote{It could be said that the irrational number was, until well into the nineteenth century, a paramathematical object, in the sense introduced by \cite{chevallard1985}.}. Consequently, no specific didactic aims were proposed with respect to the study of real numbers. These numbers were not studied; they were simply employed in problems of infinitesimal calculus \citep{edwards1979}. Thus, the didactic paradigm antecedent to the $\text{CDP}_{\text{DM}}(\mathbb{R})$ is more appropriately described as an $\text{ADP}_{\text{DM}}(\text{MA})$, where the field of knowledge MA refers to mathematical analysis. 

As regards the antecedent epistemological model, $\text{AEM}_{\text{DM}}(\text{MA})$, precisely because the real numbers were not a specific object of study, there were no types of tasks concerning, say, the properties of $(\mathbb{R},\ +,\ \cdot,\ <)$. At the core of the ontological component lay the idea of a real number as the representation of the measure of a certain magnitude of a geometrical nature, such as length or area. Thus, the corresponding epistemological component allowed arguments based on the geometrical interpretation of numbers. However, this type of argument led, within the nomological component, to the emergence of errors, paradoxes, and difficulties in proving results that appeared to be basic and straightforward. Some examples of such errors and paradoxes are: 

\begin{itemize}
	\item For much of the nineteenth century, it was erroneously accepted by the mathematical community that every continuous function was differentiable at all its points, except possibly at isolated ones (some textbooks of this period even purported to prove this claim). In this context, Karl Weierstrass's presentation in 1861 of an example of a function continuous on the entire real line but differentiable at no point caused a considerable commotion\footnote{The example was presented by Weierstrass in his lectures at the University of Berlin in 1861, although it was not published until 1872. In 1834, Bolzano had published an example of such a function, but it was overlooked by the mathematical community \citep{edwards1979}.} \citep{edwards1979}.
	
	\item On the other hand, in an article published in 1869, Charles M\'eray drew attention to a circular reasoning, a \textit{petitio principii}, in which, since the time of Cauchy, the mathematical community had been engaging: the limit of a sequence was defined as a real number, and then a real number was defined as the limit of a sequence of rational numbers \citep{boyer2011}.
	
	\item During the eighteenth century, it was taken for granted that the limit of a sequence of continuous functions was itself a continuous function \citep{lakatos1978a}. This fact was usually justified by a kind of application of a more general principle of continuity, according to which ``\textit{if a variable quantity at all stages enjoys a certain property, its limit will enjoy this same property}'' \citep[Lhuilier, 1787, as cited in][p.~153]{lakatos1978a}. Some authors, such as \citep{kline1990}, refer to this as one of Cauchy's ``errors'' in his \textit{Cours d'analyse} of 1821, when he stated (and ``proved'') that if a series of continuous functions is convergent, then its sum is also a continuous function. This statement (interpreted from the perspective of the current ontological component) is ``false''\footnote{In Section~\ref{sec7}, we will show that this ``error'' may cease to be regarded as such from other perspectives.}: as of at least 1826, a counterexample due to Abel was known; indeed, the series $
	\sum_{n=1}^{\infty}(-1)^{n+1}\frac{\sin(nx)}{n}$
	is discontinuous at $x=(2m+1)\pi$ for each natural $m$ \citep{kline1990}.
	
	These errors, together with the loss of truth status associated with Euclidean geometry triggered by the emergence of non-Euclidean geometries, led the mathematical community to reconsider the epistemological component, seeking not to base proofs involving real numbers on a geometrical interpretation of them. The following reflection by Richard Dedekind is highly revealing in this regard. At the same time, the quotation illustrates the didactic component that permeated the entire process:

	\begin{quote}
			My attention was first directed toward the considerations which form the subject of this pamphlet in the autumn of 1858. As professor in the Polytechnic School in Z\"urich I found myself for the first time obliged to lecture upon the elements of the differential calculus and felt more keenly than ever before the lack of a really scientific foundation for arithmetic. In discussing the notion of the approach of a variable magnitude to a fixed limiting value, and especially in proving the theorem that every magnitude which grows continually, but not beyond all limits, must certainly approach a limiting value, I had recourse to geometric evidence. Even now such resort to geometric intuition in a first presentation of the differential calculus, I regard as exceedingly useful, from the didactic standpoint, and indeed indispensable, if one does not wish to lose too much time. But that this form of introduction into differential calculus can make no claim to being scientific, no one will deny. For myself this feeling of dissatisfaction was so overpowering that I made the fixed resolve to keep meditating on the question till I should find a purely arithmetic and perfectly rigorous foundation for the principles of infinitesimal analysis. \citep[p. 1]{dedekind1963}. 
	\end{quote}

	Thus, we identify as $\text{AD}\varphi_{\text{DM}}(\mathbb{R})$ the errors, paradoxes, circular reasoning, and difficulties in proving intuitively obvious and seemingly basic results, due to the way in which the real numbers (and other associated notions, such as function, continuity, and differentiability) were conceived in geometrical terms.
	
	In response to the $\text{ADP}_{\text{DM}}(\mathbb{R})$, and in order to avoid the $\text{AD}\varphi_{\text{DM}}(\mathbb{R})$, the mathematical community constructed the current didactic paradigm, $\text{CDP}_{\text{DM}}(\mathbb{R})$. This paradigm was precisely the one built during the process known as\textit{ the arithmetization of analysis}. The endeavour to base analysis on elementary arithmetic, and in particular to construct $\mathbb{R}$ from the natural numbers, began in the first quarter of the nineteenth century, extended throughout the century, and was conditioned by the need to meet the new criteria of the epistemological component. The culmination of this reconstruction entailed significant changes in the ontological component with regard to the notions of real number, function, limit, continuity, and derivative, and, in the nomological component, with the emergence of new proofs of basic theorems of mathematical analysis (such as, for example, the intermediate value theorem) without appealing to the intuitive ideas of geometry \citep{kline1990}. The change in the ontological component reached its peak in 1872, the year in which four constructions of the real numbers were published simultaneously (in purely arithmetic terms), those of Dedekind, Cantor, M\'eray, and Heine \citep{edwards1979}.

\end{itemize}

\section{Formulating new didactic problems}\label{sec7}

As we have shown, the didactic analysis of a modality of study depends on the reference didactic paradigm that the researcher establishes from the outset: initially implicitly, and subsequently explicitly. It is clear that, after having completed a first didactic analysis of the current modality of study of a certain field $\mathcal{F}$ in an institution I, new didactic phenomena may emerge that had not previously been detected, thus requiring the initiation of a new investigation. This means that research will construct a new ${\text{REM}}_{I}^\prime\left(\mathcal{F}\right)$, associated with the new phenomena and providing a different conceptualisation of $\mathcal{F}$, as well as a new ${\text{RDP}}_{I}^\prime\left(\mathcal{F}\right)$. In this way, new problems of didactic research will be formulated.

In relation to the modality of study of $\mathbb{R}$ in DM, we can formulate at least two new didactic problems concerning the field of the hyperreal numbers, ${}^\ast\mathbb{R}$:

\begin{itemize}
	\item Would the construction of a ${\text{RDP}}_{\text{DM}}^\prime({}^\ast\mathbb{R})$ provide a different explanation of the antecedent phenomena? Would it bring to light other phenomena currently present in DM?
	
	\item What institutional constraints have prevented the integration of nonstandard analysis and, in particular, the construction of the hyperreal numbers, $^*\mathbb{R}$, into DM?
\end{itemize}

We shall limit ourselves to developing a few ideas regarding the first of these problems. In the description of $\text{ADP}_{\text{DM}}(\mathbb{R})$ that we have provided in Section~\ref{sec6}, we analysed the didactic phenomena that the mathematical community perceived in the past and that justified the paradigmatic shift culminating in the process known as the arithmetization of analysis. It is important to emphasize that our construction of $\text{ADP}_{\text{DM}}(\mathbb{R})$  focuses essentially on the stage immediately preceding the arithmetization of analysis and that, as we have clarified, it has been carried out from the perspective afforded by $\text{CDP}_{\text{DM}}(\mathbb{R})$ and, ultimately, by $\text{RDP}_{\text{DM}}(\mathbb{R})$. In particular, we have worked with the description of a set of real numbers satisfying the Archimedean property. However, in the stages prior to the arithmetization of analysis, the mathematical community also worked with a kind of non-Archimedean field of real numbers, which included a proto-model of the real numbers (the so-called standard real numbers), as well as an infinity of \textit{nonstandard numbers} (\textit{infinitesimals}, \textit{infinitely large numbers}, and \textit{numbers infinitely close to the standard real numbers}). In fact, one may understand that, prior to the culmination of the arithmetization of analysis, there were two ``rival'' models of the real numbers: \textit{Leibniz's theory} of non-Archimedean real numbers and \textit{Weierstrass's theory} of Archimedean real numbers \citep{lakatos1978a}.

The construction of a ${\text{RDP}}_{\text{DM}}^\prime({}^\ast\mathbb{R})$, base on \citeauthor{robinson1966}'s (\citeyear{robinson1966}) hyperreal numbers, could indeed shed light on a different explanation of the antecedent phenomena (if not all of them, then at least of some). For example, taking as a reference a didactic paradigm grounded in ${}^\ast\mathbb{R}$ could radically alter the explanation of the ``Cauchy error'' to which we referred in Section~\ref{sec6}. Indeed, Lakatos draws on \citeauthor{robinson1966}'s (\citeyear{robinson1966}) perspective to present an alternative \textit{``rational reconstruction''} of the ontological component of Cauchy's work, under which that \textit{``famous error''} ceases to be one\footnote{In \citep{laugwitz1987}, drawing on a different point of reference, yet another alternative ``rational reconstruction'' of that same ``Cauchy error'' is offered.}: \textit{``Cauchy made absolutely no mistake, he only proved a completely different theorem, about transfinite sequences or functions which Cauchy-converge on the Leibniz continuum''} \citep[p. 153]{lakatos1978a}.

\section{Conclusions: the ecology of didactic paradigm shifts}\label{sec8}

We may conclude that the main objective of this work (namely, to begin addressing the questions through which we formulated the research problem) has been reasonably achieved. Indeed:  

We have constructed a $\text{RDP}_{\text{DM}}(\mathbb{R})$ based on an alternative $\text{REM}_{\text{DM}}(\mathbb{R})$, in contrast to the $\text{CEM}_{\text{DM}}(\mathbb{R})$. The representation of $\text{CDP}_{\text{DM}}(\mathbb{R})$ obtained from the perspective of this $\text{RDP}_{\text{DM}}(\mathbb{R})$ accounts for the fact that the $\text{CD}\varphi_{\text{DM}}(\mathbb{R})$ occur as an unforeseen and unintended consequence of implementing the $\text{CDP}_{\text{DM}}(\mathbb{R})$. We posit that the modality of study modelled by $\text{RDP}_{\text{DM}}(\mathbb{R})$ would allow these $\text{CD}\varphi_{\text{DM}}(\mathbb{R})$ to be avoided, since $\text{REM}_{\text{DM}}(\mathbb{R})$ presents numbers and their operations (from the natural numbers up to the real numbers) as the result of a modelling activity necessary to address certain questions arising from working with particular magnitudes. This yields a construction of $\mathbb{R}$ that provides an explicit  raison d'\^etre for the extension of the rational numbers.

Regarding the ecological dimension, and in the absence of empirical verification, the transition from $\text{CEM}_{\text{DM}}(\mathbb{R})$ to $\text{REM}_{\text{DM}}(\mathbb{R})$ appears relatively straightforward, since the connection it proposes between $\mathbb{R}$ and the measurement of magnitudes is compatible with the current deductivist epistemological component of $\text{CEM}_{\text{DM}}(\mathbb{R})$. However, the transition from $\text{CDP}_{\text{DM}}(\mathbb{R})$\ to\ $\text{RDP}_{\text{DM}}(\mathbb{R})$, considered as a whole, would require prioritizing the problematic questions related to the measurement of magnitudes, which would entail a profound change in the modality of study (beyond the real numbers) and would impose significant institutional constraints.

Through a retrospective analysis, we have described the conditions that presumably gave rise to, and continue to sustain, the current modality of study of $\mathbb{R}$ in DM. More precisely, we have shown that the historical emergence, from the late nineteenth century, and the current persistence of the $\text{CDP}_{\text{DM}}(\mathbb{R})$ can be interpreted as a way of avoiding certain antecedent didactic phenomena, which manifested as errors, inaccuracies, and paradoxes and were therefore undesirable from the perspective of the $\text{CDP}_{\text{DM}}(\mathbb{R})$. 

Consequently, we have interpreted the current modality of study of $\mathbb{R}$ in DM from the dual perspective provided by the antecedent didactic paradigm (which explains it) and the reference didactic paradigm (which offers an alternative).

In summary, we have conducted a didactic analysis that integrates, in a coherent manner, the examination of both a hypothetical past and a prospective future of the current modality of study. Both the antecedent didactic paradigm and the reference paradigm constitute scientific hypotheses (theoretical constructs) developed to interpret possible dynamics of the current didactic paradigm.

Another fundamental contribution of this work lies in the implementation, in a concrete case, of a new general methodology for the didactic investigation of the modality of study of a particular knowledge domain, either current or potential, in a given institution. In this way, we inaugurate a methodology that links the ecology of didactic paradigm shifts, which account for the modalities of study, with the avoidance and the implementation of certain didactic phenomena (in the broad sense of study-related phenomena) that constitute the starting point of the investigation.

\bmhead{Acknowledgements}

This work has been funded by the projects PID2021-126717NB-C31, PID2021-126717NB-C32 and PID2021-126717NB-C33. Part of the research developed for this communication has been carried out during a stay of the first named author at the Centre de Recerca Matem\`atica (CRM) of Barcelona (Spain) in June 2024.

%%===========================================================================================%%
%% If you are submitting to one of the Nature Portfolio journals, using the eJP submission   %%
%% system, please include the references within the manuscript file itself. You may do this  %%
%% by copying the reference list from your .bbl file, paste it into the main manuscript .tex %%
%% file, and delete the associated \verb+\bibliography+ commands.                            %%
%%===========================================================================================%%

\bibliography{bibliographyRealNumbers}% common bib file
%% if required, the content of .bbl file can be included here once bbl is generated
%%\input sn-article.bbl

\end{document}